\documentclass{amsart}
\usepackage{amssymb}
\usepackage{amsmath}
\usepackage{amsfonts}

\newtheorem{theorem}{Theorem}
\theoremstyle{plain}

\newtheorem{corollary}{Corollary}

\newtheorem{definition}{Definition}

\newtheorem{proposition}{Proposition}

\numberwithin{equation}{section}
\input{tcilatex}

\begin{document}
\title[Integral inequalities for $P$-functions]{Some new general integral
inequalities for $P$-functions}
\author{\.{I}mdat \.{I}\c{s}can$^{\blacktriangledown }$}
\address{$^{\blacktriangledown }$Department of Mathematics, Faculty of Arts
and Sciences, Giresun University, 28100, Giresun, Turkey.}
\email{imdat.iscan@giresun.edu.tr, imdati@yahoo.com}
\author{Erhan SET$^{\clubsuit }$}
\address{$^{\clubsuit }$Department of Mathematics, Faculty of Arts and
Sciences, Ordu University, 52200, Ordu, Turkey}
\email{erhanset@yahoo.com}
\author{M. Emin \"{O}zdemir$^{\blacksquare }$}
\address{$^{\blacksquare }$Atat\"{u}rk University, K.K. Education Faculty,
Department of Mathematics, 25240, Campus, Erzurum, Turkey}
\email{emos@atauni.edu.tr}
\date{June 10, 2012}

\begin{abstract}
In this paper, we derive new estimates for the remainder term of the
midpoint, trapezoid, and Simpson formulae for functions whose derivatives in
absolute value at certain power are $P$-functions. Some applications to
special means of real numbers are also given.
\end{abstract}

\maketitle

\section{Introduction}

Let $f:I\subseteq \mathbb{R\rightarrow R}$ be a convex function defined on
the interval $I$ of real numbers and $a,b\in I$ with $a<b$. The following
inequality holds:%
\begin{equation}
f\left( \frac{a+b}{2}\right) \leq \frac{1}{b-a}\dint\limits_{a}^{b}f(x)dx%
\leq \frac{f(a)+f(b)}{2}\text{.}  \label{1-1}
\end{equation}
This double inequality is well known as Hermite-Hadamard integral inequality
for convex functions in the literature .

In \cite{DPP95} Dragomir et al. defined the concept of $P$-function as the
following:

\begin{definition}
We say that $f:I\rightarrow 
\mathbb{R}
$ is a $P$-function, or that $f$ belongs to the class $P(I)$, if $f$ is a
non-negative function and for all $x,y\in I$, $\alpha \in \lbrack 0,1]$, we
have%
\begin{equation*}
f\left( \alpha x+(1-\alpha )y\right) \leq f(x)+f(y).
\end{equation*}%
$P(I)$ contain all nonnegative monotone convex and quasi convex functions.
\end{definition}

In \cite{DPP95}, Dragomir et al., proved following inequalities of
Hadamard's type for $P$-function

\begin{theorem}
Let $f\in P(I),a,b\in I$ with $a<b$ and $f\in L\left[ a,b\right] .$ Then the
following inequality holds%
\begin{equation}
f\left( \frac{a+b}{2}\right) \leq \frac{2}{b-a}\dint\limits_{a}^{b}f(x)dx%
\leq 2\left[ f(a)+f(b)\right] \text{.}  \label{1-2}
\end{equation}
\end{theorem}

The following inequality is well known in the literature as Simpson's
inequality .

Let $f:\left[ a,b\right] \mathbb{\rightarrow R}$ be a four times
continuously differentiable mapping on $\left( a,b\right) $ and $\left\Vert
f^{(4)}\right\Vert _{\infty }=\underset{x\in \left( a,b\right) }{\sup }%
\left\vert f^{(4)}(x)\right\vert <\infty .$ Then the following inequality
holds:%
\begin{equation*}
\left\vert \frac{1}{3}\left[ \frac{f(a)+f(b)}{2}+2f\left( \frac{a+b}{2}%
\right) \right] -\frac{1}{b-a}\dint\limits_{a}^{b}f(x)dx\right\vert \leq 
\frac{1}{2880}\left\Vert f^{(4)}\right\Vert _{\infty }\left( b-a\right) ^{4}.
\end{equation*}%
\ \qquad In recent years many authors have studied error estimations for
Simpson's inequality and Hermite-Hadamard inequalitiy; for refinements,
counterparts, generalizations, see (\cite{BB11}-\cite{TYE13}).

In \cite{I12}, Iscan obtained\ a new generalization of some integral
inequalities for differentiable convex mapping which are connected Simpson
and Hadamard type inequalities by using the following lemma.

\label{2.1}Let $f:I\subseteq \mathbb{R\rightarrow R}$ be a differentiable
mapping on $I^{\circ }$ such that $f^{\prime }\in L[a,b]$, where $a,b\in I$
with $a<b$ and $\alpha ,\lambda \in \left[ 0,1\right] $. Then the following
equality holds:%
\begin{eqnarray*}
&&\lambda \left( \alpha f(a)+\left( 1-\alpha \right) f(b)\right) +\left(
1-\lambda \right) f(\alpha a+\left( 1-\alpha \right) b)-\frac{1}{b-a}%
\dint\limits_{a}^{b}f(x)dx \\
&=&\left( b-a\right) \left[ \dint\limits_{0}^{1-\alpha }\left( t-\alpha
\lambda \right) f^{\prime }\left( tb+(1-t)a\right) dt\right. \\
&&\left. +\dint\limits_{1-\alpha }^{1}\left( t-1+\lambda \left( 1-\alpha
\right) \right) f^{\prime }\left( tb+(1-t)a\right) dt\right] .
\end{eqnarray*}%
The aim of this paper is to establish some new general integral inequalities
for functions whose derivatives in absolute value at certain power are $P$%
-functions. Some applications of these results to special means is to give
as well.

\section{Main results}

Let $f:I\subseteq 
\mathbb{R}
\rightarrow 
\mathbb{R}
$ be a differentiable function on $I^{\circ }$, the interior of $I$.
Throughout this section we will take%
\begin{eqnarray*}
&&I_{f}\left( \lambda ,\alpha ,a,b\right) \\
&=&\lambda \left( \alpha f(a)+\left( 1-\alpha \right) f(b)\right) +\left(
1-\lambda \right) f(\alpha a+\left( 1-\alpha \right) b)-\frac{1}{b-a}%
\dint\limits_{a}^{b}f(x)dx
\end{eqnarray*}%
where $a,b\in I^{\circ }$ with $a<b$ and $\alpha ,\lambda \in \left[ 0,1%
\right] $.

\begin{theorem}
\label{2.2}Let $f:I\subseteq \mathbb{R\rightarrow R}$ be a differentiable
mapping on $I^{\circ }$ such that $f^{\prime }\in L[a,b]$, where $a,b\in
I^{\circ }$ with $a<b$ and $\alpha ,\lambda \in \left[ 0,1\right] $. If $%
\left\vert f^{\prime }\right\vert ^{q}$ is $P$-function on $[a,b]$, $q\geq
1, $ then the following inequality holds:%
\begin{eqnarray}
&&\left\vert I_{f}\left( \lambda ,\alpha ,a,b\right) \right\vert \leq \left(
b-a\right) \left( \left\vert f^{\prime }(b)\right\vert ^{q}+\left\vert
f^{\prime }(a)\right\vert ^{q}\right) ^{\frac{1}{q}}  \label{2-2} \\
&&\times \left\{ 
\begin{array}{cc}
\gamma _{2}(\alpha ,\lambda )+\gamma _{2}(1-\alpha ,\lambda ) & \alpha
\lambda \leq 1-\alpha \leq 1-\lambda \left( 1-\alpha \right) \\ 
\gamma _{2}(\alpha ,\lambda )+\gamma _{1}(1-\alpha ,\lambda ) & \alpha
\lambda \leq 1-\lambda \left( 1-\alpha \right) \leq 1-\alpha \\ 
\gamma _{1}(\alpha ,\lambda )+\gamma _{2}(1-\alpha ,\lambda ) & 1-\alpha
\leq \alpha \lambda \leq 1-\lambda \left( 1-\alpha \right)%
\end{array}%
\right. ,  \notag
\end{eqnarray}%
where 
\begin{eqnarray}
\gamma _{1}(\alpha ,\lambda ) &=&\left( 1-\alpha \right) \left[ \alpha
\lambda -\frac{\left( 1-\alpha \right) }{2}\right] ,  \label{2-2a} \\
\ \gamma _{2}(\alpha ,\lambda ) &=&\left( \alpha \lambda \right) ^{2}-\gamma
_{1}(\alpha ,\lambda )\ .  \notag
\end{eqnarray}
\end{theorem}

\begin{proof}
Suppose that $q\geq 1.$Since $\left\vert f^{\prime }\right\vert ^{q}$ is $P$%
-function on $[a,b]$, from Lemma \ref{2.1} and using the well known power
mean inequality, we have%
\begin{eqnarray*}
&&\left\vert I_{f}\left( \lambda ,\alpha ,a,b\right) \right\vert \\
&\leq &\left( b-a\right) \left[ \dint\limits_{0}^{1-\alpha }\left\vert
t-\alpha \lambda \right\vert \left\vert f^{\prime }\left( tb+(1-t)a\right)
\right\vert dt+\dint\limits_{1-\alpha }^{1}\left\vert t-1+\lambda \left(
1-\alpha \right) \right\vert \left\vert f^{\prime }\left( tb+(1-t)a\right)
\right\vert dt\right]
\end{eqnarray*}%
\begin{eqnarray*}
&\leq &\left( b-a\right) \left\{ \left( \dint\limits_{0}^{1-\alpha
}\left\vert t-\alpha \lambda \right\vert dt\right) ^{1-\frac{1}{q}}\left(
\dint\limits_{0}^{1-\alpha }\left\vert t-\alpha \lambda \right\vert
\left\vert f^{\prime }\left( tb+(1-t)a\right) \right\vert ^{q}dt\right) ^{%
\frac{1}{q}}\right. \\
&&\left. +\left( \dint\limits_{1-\alpha }^{1}\left\vert t-1+\lambda \left(
1-\alpha \right) \right\vert dt\right) ^{1-\frac{1}{q}}\left(
\dint\limits_{1-\alpha }^{1}\left\vert t-1+\lambda \left( 1-\alpha \right)
\right\vert \left\vert f^{\prime }\left( tb+(1-t)a\right) \right\vert
^{q}dt\right) ^{\frac{1}{q}}\right\}
\end{eqnarray*}%
\begin{equation}
\leq \left( b-a\right) \left( \left\vert f^{\prime }(b)\right\vert
^{q}+\left\vert f^{\prime }(a)\right\vert ^{q}\right) ^{\frac{1}{q}}\left\{
\dint\limits_{0}^{1-\alpha }\left\vert t-\alpha \lambda \right\vert
dt+\dint\limits_{1-\alpha }^{1}\left\vert t-1+\lambda \left( 1-\alpha
\right) \right\vert dt\right\}  \label{2-3}
\end{equation}%
Additionally, by simple computation%
\begin{equation}
\dint\limits_{0}^{1-\alpha }\left\vert t-\alpha \lambda \right\vert
dt=\left\{ 
\begin{array}{cc}
\gamma _{2}(\alpha ,\lambda ), & \alpha \lambda \leq 1-\alpha \\ 
\gamma _{1}(\alpha ,\lambda ), & \alpha \lambda \geq 1-\alpha%
\end{array}%
\right. ,  \label{2-4}
\end{equation}%
\begin{equation*}
\gamma _{1}(\alpha ,\lambda )=\left( 1-\alpha \right) \left[ \alpha \lambda -%
\frac{\left( 1-\alpha \right) }{2}\right] ,\ \gamma _{2}(\alpha ,\lambda
)=\left( \alpha \lambda \right) ^{2}-\gamma _{1}(\alpha ,\lambda )\ ,
\end{equation*}%
\begin{equation}
\dint\limits_{1-\alpha }^{1}\left\vert t-1+\lambda \left( 1-\alpha \right)
\right\vert dt=\dint\limits_{0}^{\alpha }\left\vert t-\left( 1-\alpha
\right) \lambda \right\vert dt  \label{2-5}
\end{equation}%
\begin{equation*}
=\left\{ 
\begin{array}{cc}
\gamma _{1}(1-\alpha ,\lambda ), & 1-\lambda \left( 1-\alpha \right) \leq
1-\alpha \\ 
\gamma _{2}(1-\alpha ,\lambda ), & 1-\lambda \left( 1-\alpha \right) \geq
1-\alpha%
\end{array}%
\right. ,
\end{equation*}%
Thus, using (\ref{2-4}) and (\ref{2-5}) in (\ref{2-3}), we obtain the
inequality (\ref{2-2}). This completes the proof.
\end{proof}

\begin{corollary}
Under the assumptions of Theorem \ref{2.2} with $q=1,$ we have 
\begin{eqnarray*}
&&\left\vert I_{f}\left( \lambda ,\alpha ,a,b\right) \right\vert \leq \left(
b-a\right) \left( \left\vert f^{\prime }(b)\right\vert +\left\vert f^{\prime
}(a)\right\vert \right) \\
&&\times \left\{ 
\begin{array}{cc}
\gamma _{2}(\alpha ,\lambda )+\gamma _{2}(1-\alpha ,\lambda ) & \alpha
\lambda \leq 1-\alpha \leq 1-\lambda \left( 1-\alpha \right) \\ 
\gamma _{2}(\alpha ,\lambda )+\gamma _{1}(1-\alpha ,\lambda ) & \alpha
\lambda \leq 1-\lambda \left( 1-\alpha \right) \leq 1-\alpha \\ 
\gamma _{1}(\alpha ,\lambda )+\gamma _{2}(1-\alpha ,\lambda ) & 1-\alpha
\leq \alpha \lambda \leq 1-\lambda \left( 1-\alpha \right)%
\end{array}%
\right. ,
\end{eqnarray*}
\end{corollary}

\begin{corollary}
\label{2.2a}In Theorem \ref{2.2} , if we take $\alpha =\frac{1}{2}$ and $%
\lambda =\frac{1}{3},$ then we have the following Simpson type inequality%
\begin{equation*}
\left\vert \frac{1}{6}\left[ f(a)+4f\left( \frac{a+b}{2}\right) +f(b)\right]
-\frac{1}{b-a}\dint\limits_{a}^{b}f(x)dx\right\vert \leq \frac{5\left(
b-a\right) }{36}\left( \left\vert f^{\prime }(b)\right\vert ^{q}+\left\vert
f^{\prime }(a)\right\vert ^{q}\right) ^{\frac{1}{q}}
\end{equation*}
\end{corollary}

\begin{corollary}
\label{2.2b}In Theorem \ref{2.2} , if we take $\alpha =\frac{1}{2}$ and $%
\lambda =0,$ then we have following midpoint inequality%
\begin{equation*}
\left\vert f\left( \frac{a+b}{2}\right) -\frac{1}{b-a}\dint%
\limits_{a}^{b}f(x)dx\right\vert \leq \frac{b-a}{4}\left( \left\vert
f^{\prime }(b)\right\vert ^{q}+\left\vert f^{\prime }(a)\right\vert
^{q}\right) ^{\frac{1}{q}}
\end{equation*}
\end{corollary}

\begin{corollary}
\label{2.2c}In Theorem \ref{2.2} , if we take $\alpha =\frac{1}{2}$ , and $%
\lambda =1,$ then we get the following trapezoid inequality which is
identical to the inequality in \cite[Theorem 2.3]{BB11}.%
\begin{equation*}
\left\vert \frac{f\left( a\right) +f\left( b\right) }{2}-\frac{1}{b-a}%
\dint\limits_{a}^{b}f(x)dx\right\vert \leq \frac{b-a}{4}\left( \left\vert
f^{\prime }(b)\right\vert ^{q}+\left\vert f^{\prime }(a)\right\vert
^{q}\right) ^{\frac{1}{q}}
\end{equation*}
\end{corollary}

Using Lemma \ref{2.1} we shall give another result for convex functions as
follows.

\begin{theorem}
\label{2.3}Let $f:I\subseteq \mathbb{R\rightarrow R}$ be a differentiable
mapping on $I^{\circ }$ such that $f^{\prime }\in L[a,b]$, where $a,b\in
I^{\circ }$ with $a<b$ and $\alpha ,\lambda \in \left[ 0,1\right] $. If $%
\left\vert f^{\prime }\right\vert ^{q}$ is $P$-function on $[a,b]$, $q>1,$
then the following inequality holds:%
\begin{equation}
\left\vert I_{f}\left( \lambda ,\alpha ,a,b\right) \right\vert \leq \left(
b-a\right) \left( \frac{1}{p+1}\right) ^{\frac{1}{p}}  \label{2-12}
\end{equation}%
\begin{equation*}
\times \left\{ 
\begin{array}{cc}
\left[ \varepsilon _{1}^{1/p}(\alpha ,\lambda ,p)c_{f}^{1/q}(\alpha
,q)+\varepsilon _{1}^{1/p}(1-\alpha ,\lambda ,p)k_{f}^{1/q}(\alpha ,q)\right]
, & \alpha \lambda \leq 1-\alpha \leq 1-\lambda \left( 1-\alpha \right) \\ 
\left[ \varepsilon _{1}^{1/p}(\alpha ,\lambda ,p)c_{f}^{1/q}(\alpha
,q)+\varepsilon _{2}^{1/p}(1-\alpha ,\lambda ,p)k_{f}^{1/q}(\alpha ,q)\right]
, & \alpha \lambda \leq 1-\lambda \left( 1-\alpha \right) \leq 1-\alpha \\ 
\left[ \varepsilon _{2}^{1/p}(\alpha ,\lambda ,p)c_{f}^{1/q}(\alpha
,q)+\varepsilon _{1}^{1/p}(1-\alpha ,\lambda ,p)k_{f}^{1/q}(\alpha ,q)\right]
, & 1-\alpha \leq \alpha \lambda \leq 1-\lambda \left( 1-\alpha \right)%
\end{array}%
\right. ,
\end{equation*}%
where 
\begin{eqnarray}
c_{f}(\alpha ,q) &=&\left( 1-\alpha \right) \left[ \left\vert f^{\prime
}\left( \left( 1-\alpha \right) b+\alpha a\right) \right\vert
^{q}+\left\vert f^{\prime }\left( a\right) \right\vert ^{q}\right] ,
\label{2-12a} \\
k_{f}(\alpha ,q) &=&\alpha \left[ \left\vert f^{\prime }\left( \left(
1-\alpha \right) b+\alpha a\right) \right\vert ^{q}+\left\vert f^{\prime
}\left( b\right) \right\vert ^{q}\right] ,  \notag
\end{eqnarray}%
\begin{equation}
\varepsilon _{1}(\alpha ,\lambda ,p)=\left( \alpha \lambda \right)
^{p+1}+\left( 1-\alpha -\alpha \lambda \right) ^{p+1},\ \varepsilon
_{2}(\alpha ,\lambda ,p)=\left( \alpha \lambda \right) ^{p+1}-\left( \alpha
\lambda -1+\alpha \right) ^{p+1},  \notag
\end{equation}%
and $\frac{1}{p}+\frac{1}{q}=1.$
\end{theorem}

\begin{proof}
Since $\left\vert f^{\prime }\right\vert ^{q}$ is $P$-function on $[a,b]$,
from Lemma \ref{2.1} and by H\"{o}lder's integral inequality, we have%
\begin{eqnarray*}
&&\left\vert I_{f}\left( \lambda ,\alpha ,a,b\right) \right\vert \\
&\leq &\left( b-a\right) \left[ \dint\limits_{0}^{1-\alpha }\left\vert
t-\alpha \lambda \right\vert \left\vert f^{\prime }\left( tb+(1-t)a\right)
\right\vert dt+\dint\limits_{1-\alpha }^{1}\left\vert t-1+\lambda \left(
1-\alpha \right) \right\vert \left\vert f^{\prime }\left( tb+(1-t)a\right)
\right\vert dt\right]
\end{eqnarray*}%
\begin{eqnarray}
&\leq &\left( b-a\right) \left\{ \left( \dint\limits_{0}^{1-\alpha
}\left\vert t-\alpha \lambda \right\vert ^{p}dt\right) ^{\frac{1}{p}}\left(
\dint\limits_{0}^{1-\alpha }\left\vert f^{\prime }\left( tb+(1-t)a\right)
\right\vert ^{q}dt\right) ^{\frac{1}{q}}\right.  \label{2-13} \\
&&+\left. \left( \dint\limits_{1-\alpha }^{1}\left\vert t-1+\lambda \left(
1-\alpha \right) \right\vert ^{p}dt\right) ^{\frac{1}{p}}\left(
\dint\limits_{1-\alpha }^{1}\left\vert f^{\prime }\left( tb+(1-t)a\right)
\right\vert ^{q}dt\right) ^{\frac{1}{q}}\right\} .  \notag
\end{eqnarray}%
By the inequality (\ref{1-2}), we get 
\begin{eqnarray}
\dint\limits_{0}^{1-\alpha }\left\vert f^{\prime }\left( tb+(1-t)a\right)
\right\vert ^{q}dt &=&\left( 1-\alpha \right) \left[ \frac{1}{\left(
1-\alpha \right) \left( b-a\right) }\dint\limits_{a}^{\left( 1-\alpha
\right) b+\alpha a}\left\vert f^{\prime }\left( x\right) \right\vert ^{q}dx%
\right]  \notag \\
&\leq &\left( 1-\alpha \right) \left[ \left\vert f^{\prime }\left( \left(
1-\alpha \right) b+\alpha a\right) \right\vert ^{q}+\left\vert f^{\prime
}\left( a\right) \right\vert ^{q}\right] .  \label{2-14}
\end{eqnarray}%
The inequality (\ref{2-14}) also holds for $\alpha =1$. Similarly, for $%
\alpha \in \left( 0,1\right] $ by the inequality (\ref{1-2}), we have 
\begin{eqnarray}
\dint\limits_{1-\alpha }^{1}\left\vert f^{\prime }\left( tb+(1-t)a\right)
\right\vert ^{q}dt &=&\alpha \left[ \frac{1}{\alpha \left( b-a\right) }%
\dint\limits_{\left( 1-\alpha \right) b+\alpha a}^{b}\left\vert f^{\prime
}\left( x\right) \right\vert ^{q}dx\right]  \notag \\
&\leq &\alpha \left[ \left\vert f^{\prime }\left( \left( 1-\alpha \right)
b+\alpha a\right) \right\vert ^{q}+\left\vert f^{\prime }\left( b\right)
\right\vert ^{q}\right] .  \label{2-15}
\end{eqnarray}%
The inequality (\ref{2-15}) also holds for $\alpha =0$. By simple computation%
\begin{equation}
\dint\limits_{0}^{1-\alpha }\left\vert t-\alpha \lambda \right\vert
^{p}dt=\left\{ 
\begin{array}{cc}
\frac{\left( \alpha \lambda \right) ^{p+1}+\left( 1-\alpha -\alpha \lambda
\right) ^{p+1}}{p+1}, & \alpha \lambda \leq 1-\alpha \\ 
\frac{\left( \alpha \lambda \right) ^{p+1}-\left( \alpha \lambda -1+\alpha
\right) ^{p+1}}{p+1}, & \alpha \lambda \geq 1-\alpha%
\end{array}%
\right. ,  \label{2-16}
\end{equation}%
and%
\begin{equation}
\dint\limits_{1-\alpha }^{1}\left\vert t-1+\lambda \left( 1-\alpha \right)
\right\vert ^{p}dt=\left\{ 
\begin{array}{cc}
\frac{\left[ \lambda \left( 1-\alpha \right) \right] ^{p+1}+\left[ \alpha
-\lambda \left( 1-\alpha \right) \right] ^{p+1}}{p+1}, & 1-\alpha \leq
1-\lambda \left( 1-\alpha \right) \\ 
\frac{\left[ \lambda \left( 1-\alpha \right) \right] ^{p+1}-\left[ \lambda
\left( 1-\alpha \right) -\alpha \right] ^{p+1}}{p+1}, & 1-\alpha \geq
1-\lambda \left( 1-\alpha \right)%
\end{array}%
\right. ,  \label{2-17}
\end{equation}%
thus, using (\ref{2-14})-(\ref{2-17}) in (\ref{2-13}), we obtain the
inequality (\ref{2-12}). This completes the proof.
\end{proof}

\begin{corollary}
\label{2.3a}In Theorem \ref{2.3}, if we take $\alpha =\frac{1}{2}$ and $%
\lambda =\frac{1}{3}$, then we have the following Simpson type inequality 
\begin{equation*}
\left\vert \frac{1}{6}\left[ f(a)+4f\left( \frac{a+b}{2}\right) +f(b)\right]
-\frac{1}{b-a}\dint\limits_{a}^{b}f(x)dx\right\vert \leq \frac{b-a}{12}%
\left( \frac{1+2^{p+1}}{3\left( p+1\right) }\right) ^{\frac{1}{p}}
\end{equation*}%
\begin{equation*}
\times \left\{ \left( \left\vert f^{\prime }\left( \frac{a+b}{2}\right)
\right\vert ^{q}+\left\vert f^{\prime }\left( a\right) \right\vert
^{q}\right) ^{\frac{1}{q}}+\left( \left\vert f^{\prime }\left( \frac{a+b}{2}%
\right) \right\vert ^{q}+\left\vert f^{\prime }\left( b\right) \right\vert
^{q}\right) ^{\frac{1}{q}}\right\} .
\end{equation*}
\end{corollary}

\begin{corollary}
\label{2.3b}In Theorem \ref{2.3}, if we take $\alpha =\frac{1}{2}$ and $%
\lambda =0,$ then we have the following midpoint inequality%
\begin{eqnarray*}
&&\left\vert f\left( \frac{a+b}{2}\right) -\frac{1}{b-a}\dint%
\limits_{a}^{b}f(x)dx\right\vert \leq \frac{b-a}{4}\left( \frac{1}{p+1}%
\right) ^{\frac{1}{p}} \\
&&\times \left\{ \left( \left\vert f^{\prime }\left( \frac{a+b}{2}\right)
\right\vert ^{q}+\left\vert f^{\prime }\left( a\right) \right\vert
^{q}\right) ^{\frac{1}{q}}+\left( \left\vert f^{\prime }\left( \frac{a+b}{2}%
\right) \right\vert ^{q}+\left\vert f^{\prime }\left( b\right) \right\vert
^{q}\right) ^{\frac{1}{q}}\right\} .
\end{eqnarray*}%
We note that by inequality 
\begin{equation*}
\left\vert f^{\prime }\left( \frac{a+b}{2}\right) \right\vert ^{q}\leq
\left\vert f^{\prime }\left( a\right) \right\vert ^{q}+\left\vert f^{\prime
}\left( b\right) \right\vert ^{q}
\end{equation*}%
we have%
\begin{eqnarray*}
&&\left\vert f\left( \frac{a+b}{2}\right) -\frac{1}{b-a}\dint%
\limits_{a}^{b}f(x)dx\right\vert \\
&\leq &\frac{b-a}{4}\left( \frac{1}{p+1}\right) ^{\frac{1}{p}}\left\{ \left(
\left\vert f^{\prime }\left( b\right) \right\vert ^{q}+2\left\vert f^{\prime
}\left( a\right) \right\vert ^{q}\right) ^{\frac{1}{q}}+\left( \left\vert
f^{\prime }\left( a\right) \right\vert ^{q}+2\left\vert f^{\prime }\left(
b\right) \right\vert ^{q}\right) ^{\frac{1}{q}}\right\} .
\end{eqnarray*}
\end{corollary}

\begin{corollary}
\label{2.3c}In Theorem \ref{2.3}, if we take $\alpha =\frac{1}{2}$ and $%
\lambda =1,$ then we have the following trapezoid inequality 
\begin{eqnarray*}
&&\left\vert \frac{f\left( a\right) +f\left( b\right) }{2}-\frac{1}{b-a}%
\dint\limits_{a}^{b}f(x)dx\right\vert \leq \frac{b-a}{4}\left( \frac{1}{p+1}%
\right) ^{\frac{1}{p}} \\
&&\times \left\{ \left( \left\vert f^{\prime }\left( \frac{a+b}{2}\right)
\right\vert ^{q}+\left\vert f^{\prime }\left( a\right) \right\vert
^{q}\right) ^{\frac{1}{q}}+\left( \left\vert f^{\prime }\left( \frac{a+b}{2}%
\right) \right\vert ^{q}+\left\vert f^{\prime }\left( b\right) \right\vert
^{q}\right) ^{\frac{1}{q}}\right\} .
\end{eqnarray*}
\end{corollary}

\section{Some applications for special means}

We now recall the following well-known concepts. For arbitrary real numbers $%
a,b$, $a\neq b,$ we define

\begin{enumerate}
\item The unweighted arithmetic mean%
\begin{equation*}
A\left( a,b\right) :=\frac{a+b}{2},~a,b\in 
\mathbb{R}
.
\end{equation*}

\item Then $n-$Logarithmic mean%
\begin{equation*}
L_{n}\left( a,b\right) :=\ \left( \frac{b^{n+1}-a^{n+1}}{(n+1)(b-a)}\right)
^{\frac{1}{n}}\ ,\ n\in 
\mathbb{N}
,\;n\geq 1,\;a,b\in 
\mathbb{R}
,\ a<b.
\end{equation*}
\end{enumerate}

Now we give some applications of the new results derived in section 2 to
special means of real numbers.

\begin{proposition}
Let $a,b\in 
\mathbb{R}
$ with $a<b\ $ and \ $n\in 
\mathbb{N}
,n\geq 2.$ Then 
\begin{equation*}
\left\vert \frac{1}{3}A(a^{n},b^{n})+\frac{2}{3}A^{n}(a,b)-L_{n}^{n}(a,b)%
\right\vert \leq \frac{5n\left( b-a\right) }{36}\left( \left\vert
b\right\vert ^{(n-1)q}+\left\vert a\right\vert ^{(n-1)q}\right) ^{\frac{1}{q}%
}
\end{equation*}
\end{proposition}

\begin{proof}
The assertion follows from Corollary \ref{2.2a} applied to the function $%
f(x)=x^{n}$,$\ x\in 
\mathbb{R}
$, because $\left\vert f^{\prime }\right\vert ^{q}$ is a $P$-function.
\end{proof}

\begin{proposition}
Let $a,b\in 
\mathbb{R}
$ with $a<b\ $ and \ $n\in 
\mathbb{N}
,n\geq 2.$ Then%
\begin{equation*}
\left\vert A^{n}(a,b)-L_{n}^{n}(a,b)\right\vert \leq \frac{n\left(
b-a\right) }{4}\left( \left\vert b\right\vert ^{(n-1)q}+\left\vert
a\right\vert ^{(n-1)q}\right) ^{\frac{1}{q}}
\end{equation*}%
and%
\begin{equation*}
\left\vert A(a^{n},b^{n})-L_{n}^{n}(a,b)\right\vert \leq \frac{n\left(
b-a\right) }{4}\left( \left\vert b\right\vert ^{(n-1)q}+\left\vert
a\right\vert ^{(n-1)q}\right) ^{\frac{1}{q}}
\end{equation*}
\end{proposition}

\begin{proof}
The assertion follows from Corollary \ref{2.2b} and Corollary \ref{2.2c}
applied to the function $f(x)=x^{n}$,$\ x\in 
\mathbb{R}
$, because $\left\vert f^{\prime }\right\vert ^{q}$ is a $P$-function.
\end{proof}

\begin{proposition}
Let $a,b\in 
\mathbb{R}
$ with $a<b\ $ and \ $n\in 
\mathbb{N}
,n\geq 2.$ Then 
\begin{equation*}
\left\vert \frac{1}{3}A(a^{n},b^{n})+\frac{2}{3}A^{n}(a,b)-L_{n}^{n}(a,b)%
\right\vert \leq \frac{n\left( b-a\right) }{12}\left( \frac{1+2^{p+1}}{%
3\left( p+1\right) }\right) ^{\frac{1}{p}}
\end{equation*}%
\begin{equation*}
\times \left\{ \left( \left\vert A(a,b)\right\vert ^{(n-1)q}+\left\vert
a\right\vert ^{(n-1)q}\right) ^{\frac{1}{q}}+\left( \left\vert
A(a,b)\right\vert ^{(n-1)q}+\left\vert b\right\vert ^{(n-1)q}\right) ^{\frac{%
1}{q}}\right\} .
\end{equation*}
\end{proposition}

\begin{proof}
The assertion follows from Corollary \ref{2.3a} applied to the function $%
f(x)=x^{n}$,$\ x\in 
\mathbb{R}
$, because $\left\vert f^{\prime }\right\vert ^{q}$ is a $P$-function.
\end{proof}

\begin{proposition}
Let $a,b\in 
\mathbb{R}
$ with $a<b\ $ and \ $n\in 
\mathbb{N}
,n\geq 2.$ Then%
\begin{eqnarray*}
&&\left\vert A^{n}(a,b)-L_{n}^{n}(a,b)\right\vert \leq \frac{n\left(
b-a\right) }{4}\left( \frac{1}{p+1}\right) ^{\frac{1}{p}} \\
&&\times \left\{ \left( \left\vert A(a,b)\right\vert ^{(n-1)q}+\left\vert
a\right\vert ^{(n-1)q}\right) ^{\frac{1}{q}}+\left( \left\vert
A(a,b)\right\vert ^{(n-1)q}+\left\vert b\right\vert ^{(n-1)q}\right) ^{\frac{%
1}{q}}\right\} .
\end{eqnarray*}%
and%
\begin{eqnarray*}
&&\left\vert A(a^{n},b^{n})-L_{n}^{n}(a,b)\right\vert \leq \frac{b-a}{4}%
\left( \frac{1}{p+1}\right) ^{\frac{1}{p}} \\
&&\times \left\{ \left( \left\vert A(a,b)\right\vert ^{(n-1)q}+\left\vert
a\right\vert ^{(n-1)q}\right) ^{\frac{1}{q}}+\left( \left\vert
A(a,b)\right\vert ^{(n-1)q}+\left\vert b\right\vert ^{(n-1)q}\right) ^{\frac{%
1}{q}}\right\} .
\end{eqnarray*}
\end{proposition}

\begin{proof}
The assertion follows from Corollary \ref{2.3b} and Corollary \ref{2.3c}
applied to the function $f(x)=x^{n}$,$\ x\in 
\mathbb{R}
$, because $\left\vert f^{\prime }\right\vert ^{q}$ is a $P$-function.
\end{proof}

\end{document}